\documentclass[11pt,a4paper]{article}
\usepackage{xr-hyper}
\ifx\pdfpageheight\undefined
   \usepackage[colorlinks=true,linkcolor=blue,citecolor=red,%
      urlcolor=green]{hyperref}
   \usepackage{graphicx}
   \makeatletter
   \edef\Gin@extensions{\Gin@extensions,.mps}
   \makeatother
\else
   \usepackage[pdftex]{graphicx}
   \usepackage[bookmarksopen=false,pdftex=true,breaklinks=true,%
      backref=page,pagebackref=true,plainpages=false,%
      hyperindex=true,pdfstartview=FitH,%
      colorlinks=true,%
      pdfpagelabels=true,linkcolor=blue,%
      citecolor=red,urlcolor=blue,hypertexnames=false%
      ]%
   {hyperref}
\fi

\newcommand{\som}{\sum\nolimits}

\newcommand{\ints}{\mathbb{Z}}

\newcommand{\Tr}[2]{#1_{\{#2\}}}
\newcommand{\tr}{\mbox {\sf tr}}

\usepackage{amsfonts}
\usepackage{amssymb}
\usepackage{amsmath}
\usepackage[utf8]{inputenc}

\usepackage{a4wide}
\usepackage{verbatim}
\usepackage{proof}
\usepackage{latexsym}

\usepackage{mytheorems}
\usepackage[authoryear]{natbib}

\newcommand{\pair}[1]{{\langle #1 \rangle}}

\begin{document}

\title{Some remarks about normal rings}

\author{Thierry Coquand\thanks{Department of Computer Science, Chalmers, University of G\"oteborg, 41296 G\"oteborg, Sweden. {\tt Thierry.Coquand@cse.gu.se}}~ 
and Henri Lombardi\thanks{Laboratoire de Math\'ematiques. Universit\'e de Franche-Comt\'e. F-25030 Besan{\c c}on Cedex.  {\tt henri.lombardi@univ-fcomte.fr}}}

\date{}
\maketitle

\noindent This paper has been published in
the book\\ \textsl{Concepts of proof in mathematics, philosophy, and computer
              science},\\
{Ontos Math. Log.},
{vol. 6}, {141--149},
  {De Gruyter, Berlin},
{(2016)}.
 
%:     abstract
\begin{abstract} 
We give a constructive proof that $R[X]$ is normal when $R$ is normal.
We apply this result to an operation needed for studying the henselization of a local ring.
Our proof is based on the case where $R$ is without zero divisors, which is more involved than the case where $R$ is an integral domain. We have to use a constructive deciphering technique  that replaces the use of minimal primes (in classical mathematics) by suitable explicit localizations in a suitable tree.   
\end{abstract}

\smallskip Kewwords: Normal ring, $pf$-ring, constructive mathematics, gcd-tree.

\smallskip MSC2010: 13B22, 13B30, 13B40, 03F65

\medskip  An {\em integrally closed domain} $R$ is an integral domain whose integral
closure in its field of fraction is $R$ itself. An element $b$ is {\em integral} over an ideal $I$
iff $b$ satisfies an integral relation
$$
b^n + u_1b^{n-1} + \cdots + u_n = 0
$$
with $u_l$ in $I^l$. We can reformulate the definition of being integrally closed
by stating that whenever $b$ is integral over $\pair{a}$ then $b$ belongs
to $\pair{a}$. In this form, this definition makes sense even if $R$ is an arbitrary ring
(not necessarily a domain) and this characterizes the notion of {\em normal} ring.
It can be checked that this is equivalent to the following: any localization of $R$ at
a prime ideal is an integrally closed integral domain \cite[Proposition~6.4]{DLQS}.
%: reference plus precise 

 This paper is mainly concerned with the  analysis of the
following classical result: if $R$ is an integrally closed domain then so is $R[X]$.
We first recall a proof which reduces this result to Kronecker's Theorem \cite[Theorem~3.3]{LQ}.
Interestingly, the argument depends in a crucial how we interpret constructively
the notion of ``integral domain''. Logically, to be an integral domain can be stated as
$$
\forall x.\forall y.~xy = 0 \rightarrow [x=0\vee y=0]\leqno{(1)}
$$
which is classically, but not constructively, equivalent to
$$
\forall x.~x= 0\vee [\forall y. xy = 0\rightarrow y=0].\leqno{(2)}
$$
On this form, this means that any element is $0$ or is regular. 
This Definition (2) is actually the usual definition of integral domain in constructive mathematics \cite{LQ,MRR}.
With this definition the argument using Kronecker's Theorem makes sense constructively.

The definition
(1) also has been considered in constructive algebra: a ring satisfying
this condition is called a ring {\em without zero divisors} \cite{LQ}. 
The main part of this paper presents a proof that if $R$ is a normal ring without zero divisors than so is $R[X]$. 
What is surprising is that this proof seems to require a technique which is used for analyzing argument involving {\em minimal} prime ideal \cite[Section XV-7]{LQ}. 
Furthermore, the proof involves the introduction of the notion of {\em gcd tree} of two polynomials, which is important in other context \cite{ZMT}.
Going from Definition (2) to Definition (1), classically equivalent, requires a much more complex argument. 

The advantage of Definition (2) is that it is now relatively easy to conclude from this that, more generally, if $R$ is normal (without any integrality condition) then so is $R[X]$. 
The last section analyzes a connected operation useful for studying the henselization of a local ring.

\section{Constructible and Gcd trees}

Given a reduced ring $R$ we define the notion of {\em constructible tree} for $R$. This is a binary tree.
To each node of this tree is associated a reduced ring, and $R$ is associated to the root of the tree.
Such a tree can only grow in the following way: we choose a leaf, and an element $a$ of the ring~$S$ associated to this leaf. 
We add then two sons to this node: to the left branch we associate
the ring $S[1/a]$ and to the right branch the ring $S/\sqrt{\pair{a}}$. Any such tree defines a partition of the constructible spectrum of $R$ \cite{johnstone}.

If we look at the leftmost branch of this tree, the leaf is of the form $R[1/(a_1\cdots a_n)]$ and so is a localization of the ring $R$. 

The main proofs in this note will be by induction on the size of a given constructible tree. 

\medskip
If we have two polynomials $P$ and $Q$ in $R[X]$ we can associate a constructible tree which corresponds to the formal computation of the gcd of $P$ and $Q$.
To each leaf $S$ are also associated polynomials $A,B,G,P_1,Q_1$ in $S[X]$, with $G$ monic, which witness the computation of the gcd of $P$ and $Q$
$$
AP_1+BQ_1 = 1~~~~~~P = GP_1~~~~~Q = G Q_1.
$$

Notice that for building this tree, $R$ does not need to be discrete (i.e. to have a decidable equality). 
Here is a simple example: $P = X^2$ and $Q = aX + b$. 
We start by the two branches $S_0 = R[1/a]$ and $S_1 = R/\sqrt{\pair{a}}$. 
Over $S_0$ we have the two branches $S_{00} = S_0[1/b]$ and $S_{01} = S_0/\sqrt{\pair{b}}$.
Over $S_1$ we have the two branches $S_{10} = S_1[1/b]$ and $S_{11} = S_1/\sqrt{\pair{b}}$.
The gcd is $1$ over $S_{00}$ and $S_{10}$, and is $X$ over $S_{01}$ and is $X^2$ over $S_{11}$.

This tree is called the {\em gcd tree} of $P$ and $Q$.

If one of the polynomial is monic, one can reduce the size of this tree by using subresultants~\cite{subres}.

\section{Kronecker's Theorem}

We shall only need a simple case of Kronecker's Theorem \cite[Theorem~3.3]{LQ}.

\begin{theorem}\label{kronecker}
Let $R$ be a ring, if $X^m + a_1 X^{m-1} + \cdots + a_m$ divides a polynomial of the form
$X^n + b_1X^{n-1} + \cdots + b_n$ in $R[X]$ then $a_1,\dots,a_m$ are integral over the subring
of $R$ generated by $b_1,\dots,b_n$.
\end{theorem}

\begin{proof}
We introduce the splitting algebra\footnote{$T=R[X_1,\dots,X_n]/J(f)=R[x_1,\dots,x_n]$
where $J(f)$ is the ideal of symmetric relators necessary to 
identify $\prod_{j=1}^n(X-x_j)$ with $f(X)$  in  $T[X]$. We let $x=x_1$ and the quotient ring $R[x]=R[X]/\pair f$ is identified with a subring of $T$. If $g(X,Y)=\frac {f(X)-f(Y)}{X-Y}$ then $g(x_1,X)=\prod_{i=2}^n(X-x_i)$ in $T[X]$ and $g(x_1,x_j)=0$ for $j\geq 2$.} $T$ of  $X^n + b_1X^{n-1} + \cdots + b_n$ \cite{LQ} so that
$X^n + b_1X^{n-1} + \cdots + b_n = (X-t_1)\cdots (X-t_n)$ in $T[X]$. The ring $R$ embeds in $T$
and $a_i$ is a polynomial in $t_1,\dots,t_n$.
\end{proof}

\begin{corollary}\label{kronecker1}
Let $R$ be a ring, if $Y + a_0 X^{m} + \cdots + a_m$ divides a polynominal of the form
$Y^n + b_1Y^{n-1} + \cdots + b_n$ in $R[X,Y]$ then all coefficients $a_0,\dots,a_m$ 
are roots of polynomials of the form $Z^l + p_1Z^{l-1} + \cdots + p_l$ where
$p_i$ is a homogeneous polynomial of degree $i$ in $b_1,\dots,b_n$ where $b_j$ has
weight $j$.
\end{corollary}

\begin{proof}
It is enough to look at the case where $a_0,\dots,a_m$ are indeterminates, and $R$ is
a polynomial ring on $a_0,\dots,a_m$ and some other indeterminates.
By replacing $Y$ by $X^N$ for $N$ big enough, we get that each $a_0,\dots,a_m$
is integral over $\ints[b_1,\dots,b_n]$ and hence each $a_k$ is root
of a polynomial of the form $Z^l + p_1Z^{l-1} + \cdots + p_l$ where $p_i$ is a polynomial
in $b_1,\dots,b_n$. By replacing $Y$ by $Y/c$ where $c$ is another
indeterminate, we get that $p_i$ is homogeneous of degree $i$ in $b_1,\dots,b_n$ where $b_j$ has
weight $j$.
\end{proof}

\begin{corollary}
Let $R$ be a normal integral domain, then $R[X]$ is a normal integral domain.
\end{corollary}

\begin{proof}
We assume given $P$ and $Q$ in $R[X]$ such
that $P$ is integral over $\pair{Q}$ and we want to show that
$P$ is in $\pair{Q}$ in $R[X]$. Let $K$ be the total fraction field of $R$.
Since $R$ is an integral domain, we can consider $R$ to be a subring of $K$.
Since $K[X]$ is euclidean, we know that $P$ is in $\pair{Q}$ in $K[X]$
and we have $cP = HQ$ for some regular element $c$.
Since $P$ is integral over $\pair{Q}$ we have a relation
$$
P^n + A_1QP^{n-1} + \cdots + A_n Q^n = 0
$$
with $A_1,\dots,A_n$ in $R[X]$ and so 
we can write
$$
Y^n + A_1 c Y^{n-1} + \cdots + A_n c^n = (Y - H)S(X,Y)
$$
where $S(X,Y)$ is a monic polynomial in $R[X][Y]$. Using Corollary \ref{kronecker1},
it follows that all coefficients of $H$ are integral over $\pair{c}$ and
hence are in $\pair{c}$ since $R$ is normal. We can then write $H = cH_1$
and so $c(P-QH_1) = 0$ in $R[X]$. It follows that we have $P=QH_1$ and hence $P$ is in $\pair{Q}$ in $R[X]$.
\end{proof}

 This is the argument we are going to adapt in the case where $R$ is normal
and without zero divisors.

\section{Polynomial ring}

 We assume that $R$ is normal without zero divisors and we show
that $R[X]$ is normal. We assume given $P$ and $Q$ in $R[X]$ such
that $P$ is integral over $\pair{Q}$ and we want to show that
$P$ is in $\pair{Q}$ in $R[X]$.

\begin{lemma}\label{reduced}
If $R$ is normal then $R$ is reduced.
\end{lemma}

\begin{proof}
If $b^2 = 0$ then $b$ is integral over $\pair{0}$ and so is in $\pair{0}$.
\end{proof}

\begin{lemma}\label{locnormal}
If $R$ is normal then so is $R[1/a]$.
\end{lemma}

\begin{proof}
For $c$ and $b$ in $R$, if $c$ is integral over $\pair{b}$ in $R[1/a]$ we have a relation
$(a^Nc)^n + u_1 b (a^Nc)^{n-1} + \cdots + u_n b^n = 0$ with $u_1,\dots,u_n$
in $R$. Since $R$ is normal we have $a^Nc$ in $\pair{b}$.
\end{proof}

\begin{lemma}\label{without}
If $R$ is without zero divisors  then so is $R[1/a]$.
\end{lemma}

\begin{proof}
We take two elements $v = b/a^n$ and $w = c/a^m$ of $R[1/a]$ with $b$ and $c$ in $R$. 
If we have $vw=0$ in $R[1/a]$ we have $a^pb c = 0$ in $R$ for some $p\geqslant 0$.
We have then $a^pb = 0$ in $R$ or $a^pc = 0$ in $R$, which implies that $v=0$ or $w=0$
in $R[1/a]$.
\end{proof}

 From now on in this section, we assume $R$ to be a normal ring without zero divisors.

\begin{lemma}\label{main}
If $P$ is integral over $\pair{Q}$ and is in $\pair{Q}$ in $R[1/a][X]$ then $a = 0$ or
$P$ is in $\pair{Q}$ in $R[X]$.
\end{lemma}

\begin{proof}
We have $H$ in $R[X]$ such that $a^NP = QH$ for some $N$. We write
$c = a^N$. Since $P$ is integral over $\pair{Q}$ we have a relation
$$
P^n + A_1QP^{n-1} + \cdots + A_n Q^n = 0
$$
with $A_1,\dots,A_n$ in $R[X]$ and so 
$$
Q^n(H^n + A_1 c H^{n-1} + \cdots + A_n c^n) = 0
$$
in $R[X]$. Hence either $Q=0$ in which case $P=0$ is in $\pair{Q}$
or we can write
$$
Y^n + A_1 c Y^{n-1} + \cdots + A_n c^n = (Y - H)S(X,Y)
$$
where $S(X,Y)$ is a monic polynomial in $R[X][Y]$. Using the corollary of
Kronecker's Theorem \ref{kronecker1}, it follows that all coefficients of $H$ are integral over
$\pair{c}$ and
hence are in $\pair{c}$ since $R$ is normal. We can then write $H = cH_1$
and so $c(P-QH_1) = 0$ in $R[X]$. It follows that we have $c=0$, that is equivalent
to $a=0$, or $P=QH_1$ and hence $P$ is in $\pair{Q}$ in $R[X]$.
\end{proof}

\begin{lemma}\label{tree}
If we have $P$ and $Q$ in $R[X]$ and a constructible tree for $R$ such as, at all leaves
$S$ of this tree, we have $P$ in $\pair{Q}$ in $S[X]$. Then $P$ is in $\pair{Q}$ in $R[X]$.
\end{lemma}

\begin{proof}
 We look at the leftmost branch of this tree,
indexed by elements $a_1,\dots,a_l$, so that $S = S'[1/a_l]$ where $S' = R[1/(a_1\cdots a_{l-1})]$ is 
without zero divisors by Lemma \ref{without} and is normal by Lemma \ref{locnormal}. 
Using Lemma \ref{main}
we get that $a_l = 0$ in $S'$ or $P$ is in $\pair{Q}$ in $S'[X]$. In the second
case, we can shorten the leftmost branch to $a_1,\dots,a_l$ and get a smaller tree.
In the first case where $a_l = 0$ in $S'$, this means that the right son $S'/\pair{a}$ of $S'$
is equal to $S'$ and we also can shorten the tree. We conclude by tree induction.
\end{proof}

\begin{theorem}\label{normal}
If $R$ is normal and without zero divisors then so is $R[X]$.
\end{theorem}

\begin{proof}
We take $P$ and $Q$ in $R[X]$ and we assume that we have a relation
$$
P^n + A_1QP^{n-1} + \cdots + A_n Q^n = 0
$$
with $n\geqslant 1$ and $A_1,\cdots,A_n$ in $R[X]$. We have to show that $P$ is in $\pair{Q}$
in $R[X]$.

We look now at the gcd tree of $P$ and $Q$ as defined in the first section.
At all leaves $S$ of this tree, we have $P_1,Q_1,G,A,B$
in $S[X]$ satisfying
$$
P = GP_1,~~Q = GQ_1,~~AP_1+BQ_1 = 1
$$
in $S[X]$ and $G$ is monic. Since $G$ is monic and
$$
P^n + A_1QP^{n-1} + \cdots + A_n Q^n = G^n(P_1^n + A_1Q_1P_1^{n-1} + \cdots + A_n Q_1^n) = 0
$$
we have
$$
P_1^n + A_1Q_1P_1^{n-1} + \cdots + A_n Q_1^n = 0
$$
and $Q_1$ divides $P_1^n$. With $AP_1+BQ_1=1$ this implies that $Q_1$ is a unit
and so $P$ is in $\pair{Q}$ in $S[X]$. We can now apply Lemma \ref{tree}.
\end{proof}

\section{Normal ring}

We say that the ring is {\em locally}
without zero divisors \cite[Lemma VIII-3.2]{LQ} if, and only if, whenever $ab=0$ then there exists $u$ such that
$ua=0$ and $(1-u)b=0$. 
These rings are often called $pf$-rings. 
%: rajout terminologie
In this note, only the notion of rings without zero divisors
and locally without zero divisors will play a role.

\begin{lemma}\label{local}
If $R$ is normal then $R$ is locally without zero divisors.
\end{lemma}

\begin{proof}
If $ab = 0$ then $b^2 -(a+b)b=0$ so $b$ is integral over $\pair{a+b}$ and so is in $\pair{a+b}$.
We can write $b = (a+b)u$ and so $ua = (1-u)b$. This implies $ua^2 = (1-u)ba=0$ and so
$ua=(1-u)b=0$ since $R$ is reduced.
\end{proof}

\begin{theorem}
If $R$ is normal then so is $R[X]$.
\end{theorem}

\begin{proof}
By Lemma \ref{local}, $R$ is locally without zero divisors. Assume then that a polynomial $P\in R[X]$
is integral over $\pair{Q}$ in $R[X]$.
%: ci-dessous un ajout pour une remarque du referee 
Following the proof of Theorem \ref{normal}, each time we use $ab=0\rightarrow a=0 \hbox{ or } b=0$, we split the ``current ring $R[1/v]$'' 
in two rings $R[1/vu]$ and $R[1/v(1-u)]$ by using an $u$ such that $ua=0$ and $(1-u)b=0$. We find finally  
%Following the previous Theorem \ref{normal}, we can find
$u_1,\dots,u_m$ in~$R$ such that $\pair{u_1,\dots,u_m} = 1$ and~$P$ belongs to $\pair{Q}$
in each $R[1/u_j][X]$. It follows that $P$ is in $\pair{Q}$ as required.
\end{proof}

\section{The ring $\Tr{R}{f}$}

 Let $f$ be a monic polynomial in $R[X]$. We can consider the extension $S = R[X]/\pair{f}$ where
$f$ has a root $x$. We let $f_X$ be the formal derivative of $f$ w.r.t. $X$,
and we define $\Tr{R}{f}$ to be the localization $S[1/f_X(x)]$. This construction
is important to study the properties of henselization of a local ring \cite{raynaud}. 

 The goal of this section is to show that $\Tr{R}{f}$ is normal whenever $R$ is normal. As in the
previous section, we can first assume that $R$ is without zero divisors, and then use that a normal
ring is locally without zero divisors to conclude. So in the rest of the section, we assume that
$R$ is a normal ring without zero divisors.

 If $f = gh$ is the product of two monic polynomials $g$ and $h$
we have $\Tr{R}{f}$ isomorphic to $\Tr{R}{g}[1/h(x)]\times \Tr{R}{h}[1/g(x)]$.
This remark is important since by using Lemma \ref{locnormal} we can reason by induction on the
degree of $f$ to show that $\Tr{R}{f}$ is normal if $R$ is normal.

%If $cp(x) = 0$ with $c$ in $R$
%and $p$ in $R[X]$ then $c = 0$ in $R$ or $p(x) = 0$ in $S$? This is the case
%since we can write $p(x) = a_{m-1}x^{m-1} + \dots + a_0$ with $m = deg(f)$
%and $cp(x) = 0$ iff $c a_{m-1} = \dots = ca_0 = 0$.

\begin{lemma}\label{crucial}
If $R$ is normal without zero divisors, and $a$ in $R$ and $T = R[1/a]$
and $f = g f_1$ with $g$ and $f_1$ monic in $T[X]$
then we have $g$ and $f_1$ in $R[X]$
or $a=0$.
\end{lemma}

\begin{proof}
Using Kronecker's Theorem \ref{kronecker}, each coefficient of $g$ and $f_1$ is
integral over $R$. Since $R$ is normal and without zero divisors, this implies
that $a=0$ or $g$ and $f_1$ are in $R[X]$.
\end{proof}

 We have the trace map $\tr:S\rightarrow R$. If we introduce the splitting algebra \cite[Definition~III-4.1]{LQ} of $f$ and write $f = (X-x_1)\cdots (X-x_{n})$ with $x = x_1$ then the trace of $h(x)\in S$ is $h(x_1)+\cdots + h(x_{n})$.
If $v=h(x)$ in $S$ is integral over $\pair{a}_S$ with $a$ in $R$ then all elements $h(x_1),\dots,h(x_{n})$ are integral over $\pair{a}_R$ and so $\tr(v)$ is also integral over $\pair{a}_X$ and so is in $\pair{a}_X$ 
since~$R$ is normal. 
Also if we write $f(X) - f(Y) = (X-Y) g(X,Y) = \sum_i g_{i}(Y)X^i$, we have $f_X(T)=g(T,T)$ and $g(x_1,x_j)=0$ for $j\neq 1$. So we get for all $v=h(x)=h(x_1)$ in~$S$ 
$$
f_X(x)h(x) = g(x_1,x_1) h(x_1) = \som_i g_{i}(x_1) h(x_1)\; x_1^i
$$
and for $j\neq 1$ 
$$
0 =  g(x_1,x_j) h(x_j) = \som_i g_{i}(x_j) h(x_j)\; x_1^i
$$
so that, by summation, we get Tate's formula \cite[Chapter VII, 1]{raynaud}
$$
f_X(x) v = \som_i \tr(g_i(x)v)\; x^i.
$$
Since each $g_i(x)$ is integral over $R$, we can state the following lemma.

\begin{lemma}\label{Tate}
If $R$ is normal, if $a$ in $R$ and if $v$ in $S$ is integral over $\pair{a}_S$ then $f_X(x)v$ is in $\pair{a}_S$.
\end{lemma}

\begin{theorem}
If $R$ is normal then $\Tr{R}{f}$ is normal.
\end{theorem}

\begin{proof}
 We assume given $p,q$ in $R[X]$ such that $p(x)$ is integral over $\pair{q(x)}$ in $S$
so that we have a relation $p(x)^n + u_1(x)q(x)p(x)^{n-1} + \cdots + u_n(x)q(x)^n = 0$.
The goal is to show that $p(x)$ is in $q(x)S[f_X(x)^{-1}]$.

 We look at the gcd tree of $q$ and $f$, and the leftmost branch of this tree. 
At the leaf of this branch we have a list of elements that we force to be invertible
$a_1,\dots,a_n$ and  $q_1,f_1,g,A,B$ in $R[a^{-1}]$ with $a= a_1\dots a_n$ such that
$$
f = gf_1~~~~q=gq_1~~~~~1 = Af_1+Bq_1.
$$
Furthermore $g$ and hence $f_1$ are monic since $f$ is monic.

\medskip
 If $f = f_1$ we have $g=1$ and $q = q_1$. In this case we have $c = Af + Bq$ where $c = (a_1\dots a_n)^m$
for some $m$ and so we have $c = B(x) q(x)$ in $S = R[X]/\pair{f}$.
We then have a relation
$$
(p(x)B(x))^n + u_1(x)c(p(x)B(x))^{n-1} + \cdots + u_n(x)c^n = 0
$$
and hence, by Lemma \ref{Tate}, we get that $p(x)B(x)$ is in $\pair{c}$
in $S[f_X(x)^{-1}]$. Hence we have $l(x)$ in~$S$ and $N$
such that 
$$
f_X(x)^Np(x)B(x) = c l(x) = q(x)B(x)l(x)
$$
and so
$$
c(p(x)f_X(x)^N - l(x)q(x)) = 0.
$$
We have then $c =0$ or $p(x)$ is in  $\pair{q(x)}$ in $R[f_X(x)^{-1}]$. So either
we have the desired conclusion that $p(x)$ is in  $\pair{q(x)}$ or
we have $a_n = 0$ in $R[1/(a_1\cdots a_{n-1})]$ and we can shorten the computation
tree of the gcd of $f$ and $q$.

\medskip  If $f$ and $f_1$ have not the same degree, we have found
a proper decomposition $f = gf_1$ of $f$ in $R[1/a]$ with $a= a_1\cdots a_n$.
In this case, since $R$ is normal, by Lemma \ref{crucial}, we have two subcases
\begin{itemize}
\item either $g$ and $f_1$ are in $R[X]$
and we can conclude by induction on the degree of $f$, 
using  that $R_{\{f\}}$ isomorphic to $\Tr{R}{g}[1/f_1]\times \Tr{R}{f_1}[1/g]$, 
\item or $a=0$ 
and as in the previous case, we can shorten the computation
tree of the gcd of $f$ and $q$.
\end{itemize}
\end{proof}
% If we don't assume $R$ to be an integral domain, we reason instead with the tree
%that represents the computation of the gcd of $q$ and $f$. As before we look at the 
%leftmost branch, which is indexed by elements $a_1,\dots,a_n$. 

 As in \cite[VIII-4.4]{LQ}, we say that a ring is a Pr\"ufer ring if it is arithmetic and reduced.
A coherent Pr\"ufer ring is an arithmetic $pp$-ring.

\begin{corollary}
If $R$ is a Pr\"ufer ring of Krull dimension $\leqslant 1$ then so is $\Tr{R}{f}$.
\end{corollary}

\begin{proof}
We use the fact that a ring is normal coherent ring of Krull dimension $\leqslant 1$ if, and only if,
it is Pr\"ufer and of Krull dimension $\leqslant 1$  \cite{DLQS}.
We have shown that $\Tr{R}{f}$ is normal. Since $S = R[X]/\pair{f}$ is an integral extension of $R$
it is also of Krull dimension $\leqslant 1$ \cite{Ducos1} and so is its localization $\Tr{R}{f}$. Finally, $S$
is a finite free $R$-module, and so it is coherent if $R$ is coherent and so is its localization
$\Tr{R}{f}$.
\end{proof}

 This gives an alternative proof to the main result of \cite{Coq5}, that $\Tr{R}{f}$ is Pr\"ufer
when $R=k[X]$, in the case where $f$ is monic in $Y$. It is possible however to reduce the general case to this
case, by a change of variables.

%: biblio
\addcontentsline{toc}{section}{Références}
%%\bibliographystyle{plain}
%\bibliographystyle{plainnat}
%\bibliography{normal}

\tableofcontents

\end{document}